\documentclass[a4paper,reqno,12pt]{amsart}
\usepackage{amsfonts}
\usepackage{amsmath}
\usepackage{amssymb}
\usepackage{hyperref}
\usepackage{graphicx}

\newtheorem{theorem}{Theorem}[section]

\newtheorem{lemma}[theorem]{Lemma}

\def\func#1{\mathop{\mathrm{#1}}\nolimits}
\def\dint{\displaystyle\int}

\def\Qlb#1{#1}

\def\FRAME#1#2#3#4#5#6#7#8
{
 \begin{figure}[ptbh]
 \begin{center}
 \includegraphics[height=#3]{#7}
 \caption{#5}
 \label{#6}
 \end{center}
 \end{figure}
}

\begin{document}
\title[Schrodinger eigenvalues]{A lower bound for the number of negative
eigenvalues of Schr\"{o}dinger operators}
\author[Grigor'yan]{Alexander Grigor'yan}
\thanks{AG was supported by SFB 701 of German Research Council}
\address{Department of Mathematics, University of Bielefeld, 33501
Bielefeld, Germany\\
\texttt{grigor@math.uni-bielefeld.de}}
\author[Nadirashvili]{Nikolai Nadirashvili}
\thanks{NN was supported by the Alexander von Humboldt Foundation}
\address{Universit\'{e} Aix-Marseille, CNRS, I2M, Marseille,
France \\
\texttt{nicolas@cmi.univ-mrs.fr}}
\author[Sire]{Yannick Sire}
\address{Universit\'{e} Aix-Marseille, I2M, UMR 7353, Marseille, France\\
\texttt{sire@cmi.univ-mrs.fr}}
\maketitle

\begin{abstract}
We prove a lower bound for the number of negative eigenvalues for a Sch\"{o}%
dinger operator on a Riemannian manifold via the integral of the potential.
\end{abstract}

\section{Introduction}

Let $(M,g)$ be a compact Riemannian manifold without boundary. Consider the
following eigenvalue problem on $M$: 
\begin{equation}
-\Delta u-Vu=\lambda u,  \label{eqSchro}
\end{equation}%
where $\Delta $ is the Laplace-Beltrami operator on $M$ and $V\in L^{\infty
}\left( M\right) $ is a given potential. It is well-known, that the operator 
$-\Delta -V$ has a discrete spectrum. Denote by $\left\{ \lambda
_{k}(V)\right\} _{k=1}^{\infty }$ the sequence of all its eigenvalues
arranged in increasing order, where the eigenvalues are counted with
multiplicity.

Denote by $\mathcal{N}(V)$ the number of negative eigenvalues of (\ref%
{eqSchro}), that is,

\begin{equation*}
\mathcal{N}(V)=\func{card}\left\{ k:\,\lambda _{k}(V)<0\right\} .
\end{equation*}%
It is well-known that $\mathcal{N}\left( V\right) $ is finite. Upper bounds
of $\mathcal{N}\left( V\right) $ have received enough attention in the
literature, and for that we refer the reader to \cite{BirSol}, \cite%
{GrigNadirNeg}, \cite{Lieb}, \cite{LiYauE}, \cite{YangYau} and references
therein.

However, a little is known about lower estimates. Our main result is the
following theorem. We denote by $\mu $ the Riemannian measure on $M$.

\begin{theorem}
\label{main}Set $\dim M=n$. For any $V\in L^{\infty }\left( M\right) $ the
following inequality is true: 
\begin{equation}
\mathcal{N}(V)\geq \frac{C}{\mu \left( M\right) ^{n/2-1}}\left(
\int_{M}Vd\mu \right) _{+}^{n/2},  \label{NV}
\end{equation}%
where $C>0$ is a constant that in the case $n=2$ depends only on the genus
of $M$ and in the case $n>2$ depends only on the conformal class of $M$.
\end{theorem}

In the case $V\geq 0$ the estimate (\ref{NV}) was proved in \cite[Theorems
5.4 and Example 5.12]{GNY}. Our main contribution is the proof of (\ref{NV})
for signed potentials $V$ (as it was conjectured in \cite{GNY}), with the
same constant $C$ as in \cite{GNY}. In fact, we reduce the case of a signed $%
V$ to the case of non-negative $V$ by considering a certain variational
problem for $V$ and by showing that the solution of this problem is
non-negative. The latter method originates from \cite{NS}.

In the case $n=2$, inequality (\ref{NV}) takes the form%
\begin{equation}
\mathcal{N}\left( V\right) \geq C\int_{M}Vd\mu .  \label{NV2}
\end{equation}%
For example, the estimate (\ref{NV2}) can be used in the following
situation. Let $M$ be a two-dimensional manifold embedded in $\mathbb{R}^{3}$
and the potential $V$ be of the form $V=\alpha K+\beta H$ where $K$ is the
Gauss curvature, $H$ is the mean curvature, and $\alpha ,\beta $ are real
constants (see \cite{Harrell}, \cite{Soufi}). In this case (\ref{NV2})
yields 
\begin{equation*}
\mathcal{N}\left( V\right) \geq C\left( K_{total}+H_{total}\right) ,
\end{equation*}%
where $K_{total}$ is the total Gauss curvature and $H_{total}$ is the total
mean curvature. We expect in the future many other applications of (\ref{NV}%
)-(\ref{NV2}).

\section{A variational problem}

Fix positive integers $k,N$ and consider the following optimization problem:
find $V\in L^{\infty }\left( M\right) $ such that%
\begin{equation}
\int_{M}Vd\mu \rightarrow \max \ \text{under restrictions }\lambda
_{k}\left( V\right) \geq 0\ \text{and\ }\left\Vert V\right\Vert _{L^{\infty
}}\leq N.  \label{optim}
\end{equation}%
Clearly, the functional $V\mapsto \int_{M}Vd\mu $ is weakly continuous in $%
L^{\infty }\left( M\right) $. Since the class of potentials $V$ satisfying
the restrictions in (\ref{optim}) is bounded in $L^{\infty }\left( M\right) $%
, it is weakly precompact in $L^{\infty }\left( M\right) $. In fact, we
prove in the next lemma that this class is weakly compact, which will imply
the existence of the solution of (\ref{optim}).

\begin{lemma}
The class 
\begin{equation*}
C_{k,N}=\left\{ V\in L^{\infty }\left( M\right) :\lambda _{k}\left( V\right)
\geq 0\text{ and }\left\Vert V\right\Vert _{L^{\infty }}\leq N\right\}
\end{equation*}%
is weakly compact in $L^{\infty }\left( M\right) $. Consequently, the
problem \emph{(\ref{optim})} has a solution $V\in L^{\infty }(M)$.
\end{lemma}

\begin{proof}
It was already mentioned that the class $C_{k,N}$ is weakly precompact in $%
L^{\infty }\left( M\right) $. It remains to prove that it is weakly closed,
that is, for any sequence $\left\{ V_{i}\right\} \subset C_{k,N}$ that
converges weakly in $L^{\infty }$, the limit $V$ is also in $C_{k.N}$. The
condition $\left\Vert V\right\Vert _{L^{\infty }}\leq N$ is trivially
satisfied by the limit potential, so all we need is to prove that $\lambda
_{k}\left( V\right) \geq 0.$ Let us use the minmax principle in the
following form:%
\begin{equation*}
\lambda _{k}\left( V\right) =\inf_{\substack{ E\subset W^{1,2}\left(
M\right)  \\ \dim E=k}}\sup_{u\in E\setminus \left\{ 0\right\} }\frac{%
\int_{M}\left\vert \nabla u\right\vert ^{2}d\mu -\int_{M}Vu^{2}d\mu }{%
\int_{M}u^{2}d\mu },
\end{equation*}%
where $E$ is a subspace of $W^{1,2}\left( M\right) $ of dimension $k$. The
condition $\lambda _{k}\left( V\right) \geq 0$ is equivalent then to the
following:%
\begin{equation}
\left. 
\begin{array}{l}
\forall E\subset W^{1,2}\left( M\right) \ \text{with }\dim E=k\ \ \ \
\forall \varepsilon >0\ \ \ \exists u\in E\setminus \left\{ 0\right\}  \\ 
\text{such that }\dint_{M}\left\vert \nabla u\right\vert ^{2}d\mu
-\dint_{M}Vu^{2}d\mu \geq -\varepsilon \dint_{M}u^{2}d\mu .%
\end{array}%
\right.   \label{Eu}
\end{equation}%
Fix a subspace $E\subset W^{1,2}\left( M\right) $ of dimension $k$ and some $%
\varepsilon >0$. Since $\lambda _{k}\left( V_{i}\right) \geq 0$, we obtain
that there exists $u_{i}\in E\setminus \left\{ 0\right\} $ such that 
\begin{equation}
\dint_{M}\left\vert \nabla u_{i}\right\vert ^{2}d\mu
-\dint_{M}V_{i}u_{i}^{2}d\mu \geq -\varepsilon \dint_{M}u_{i}^{2}d\mu .
\label{Eun}
\end{equation}%
Without loss of generality we can assume that $\left\Vert u_{i}\right\Vert
_{W^{1,2}\left( M\right) }=1$. Then the sequence $\left\{ u_{i}\right\} $
lies on the unit sphere in the finite-dimensional space $E$. Hence, it has a
convergent (in $W^{1,2}\left( M\right) $-norm) subsequence. We can assume
that the whole sequence $\left\{ u_{i}\right\} $ converges in $E$ to some $%
u\in E$ with $\left\Vert u\right\Vert _{W^{1,2}\left( M\right) }=1.$ It
remains to verify that $u$ satisfies the inequality (\ref{Eu}). By
construction we have%
\begin{equation*}
\dint_{M}\left\vert \nabla u_{i}\right\vert ^{2}d\mu \rightarrow
\dint_{M}\left\vert \nabla u\right\vert ^{2}d\mu \ \ \ \ \text{and\ \ \ \ }%
\dint_{M}u_{i}^{2}d\mu \rightarrow \dint_{M}u^{2}d\mu .
\end{equation*}%
Next we have%
\begin{eqnarray*}
\left\vert \dint_{M}V_{i}u_{i}^{2}d\mu -\dint_{M}Vu^{2}d\mu \right\vert 
&\leq &\left\vert \dint_{M}\left( V_{i}u_{i}^{2}-V_{i}u^{2}\right) d\mu
\right\vert +\left\vert \dint_{M}\left( V_{i}u^{2}-Vu^{2}\right) d\mu
\right\vert  \\
&\leq &N\left\Vert u_{i}-u\right\Vert _{L^{2}}^{2}+\left\vert \int_{M}\left(
V_{i}-V\right) u^{2}d\mu \right\vert .
\end{eqnarray*}%
By construction we have $\left\Vert u_{i}-u\right\Vert _{L^{2}}\rightarrow 0$
as $i\rightarrow \infty $. Since $u^{2}\in L^{1}\left( M\right) $, the $%
L^{\infty }$ weak convergence $V_{i}\rightharpoonup V$ implies that 
\begin{equation*}
\int_{M}\left( V_{i}-V\right) u^{2}d\mu \rightarrow 0\ \ \text{as }%
i\rightarrow \infty .
\end{equation*}%
Hence, the inequality (\ref{Eu}) follows from (\ref{Eun}).
\end{proof}

\begin{lemma}
\label{Lemlak=0}If $N$ is large enough (depending on $k$ and $M$) then any
solution $V$ of \emph{(\ref{optim})} satisfies $\lambda _{k}(V)=0.$
\end{lemma}

\begin{proof}
Assume that $\lambda _{k}(V)>0$ and bring this to a contradiction. Consider
the family of potentials 
\begin{equation*}
V_{t}=(1-t)V+tN\text{\ \ where}\,\,t\in \lbrack 0,1].
\end{equation*}%
Since $V_{t}\geq V$, we have by a well-known property of eigenvalues that $%
\lambda _{k}(V_{t})\leq \lambda _{k}(V)$. By continuity we have, for small
enough $t$, that $\lambda _{k}(V_{t})>0$. Clearly, we have also $\left\vert
V_{t}\right\vert \leq N$. Hence, $V_{t}$ satisfies the restriction of the
problem (\ref{optim}), at least for small $t$. If $\mu \left\{ V<N\right\}
>0 $ then we have for all $t>0$ 
\begin{equation*}
\int_{M}V_{t}>\int_{M}V,
\end{equation*}%
which contradicts the maximality of $V$. Hence, we should have $V=N$ $%
\mathrm{a.e.}$. However, if $N>\lambda _{k}\left( -\Delta \right) $ then $%
\lambda _{k}\left( -\Delta -N\right) <0$ and $V\equiv N$ cannot be a solution
of (\ref{optim}). This contradiction finishes the proof.
\end{proof}

\section{Proof of Theorem \protect\ref{main}}

The main part of the proof of Theorem \ref{main} is contained in the
following lemma.

\begin{lemma}
\label{crucial} Let $V_{\max }$ be a maximizer of the variational problem (%
\ref{optim}). Then $V_{\max }$ satisfies the inequality 
\begin{equation*}
V_{\max }\geq 0\,\,\,\mathrm{a.e.}\,\text{on\ }\,M
\end{equation*}
\end{lemma}

\subsection{Proof of Theorem \protect\ref{main} assuming Lemma \protect\ref%
{crucial}}

Choose $N$ large enough, say 
\begin{equation*}
N>\sup_{M}|V|.
\end{equation*}%
Set $k=\mathcal{N}(V)+1$ so that $\lambda _{k}(V)\geq 0$. For the maximizer $%
V_{\max }$ of (\ref{optim}) we have 
\begin{equation*}
\int_{M}V\,d\mu \leq \int_{M}V_{\max }\,d\mu .
\end{equation*}%
On the other hand, since $V_{\max }\geq 0$, we have by \cite{GNY} 
\begin{equation*}
\mathcal{N}(V_{\max })\geq \frac{C}{\mu \left( M\right) ^{n/2-1}}\left(
\int_{M}V_{\max }\,d\mu \right) ^{n/2}.
\end{equation*}%
Also , we have 
\begin{equation*}
\lambda _{k}(V_{\max })\geq 0,
\end{equation*}%
which implies 
\begin{equation*}
\mathcal{N}(V_{\max })\leq k-1=\mathcal{N}(V).
\end{equation*}%
Hence, we obtain 
\begin{equation*}
\mathcal{N}(V)\geq \mathcal{N}(V_{\max })\geq \frac{C}{\mu \left( M\right)
^{n/2-1}}\left( \int_{M}V_{\max }\,d\mu \right) ^{n/2}\geq \frac{C}{\mu
\left( M\right) ^{n/2-1}}\left( \int_{M}V\,d\mu \right) _{+}^{n/2},
\end{equation*}%
which was to be proved.

\subsection{Some auxiliary results}

Before we can prove Lemma \ref{crucial}, we need some auxiliary lemmas. The
following lemma can be found in \cite{Kato}.

\begin{lemma}
\label{Lemlk}Let $V\left( t,x\right) $ be a function on $\mathbb{R}\times M$
such that, for any $t\in \mathbb{R}$, $V\left( t,\cdot \right) \in L^{\infty
}\left( M\right) $ and $\partial _{t}V(t,\cdot )\in L^{\infty }\left(
M\right) $. For any $t\in \mathbb{R}$, consider the Schr\"{o}dinger operator 
$L_{t}=-\Delta -V\left( t,\cdot \right) $ on $M$ and denote by $\left\{
\lambda _{k}\left( t\right) \right\} _{k=1}^{\infty }$ the sequence of the
eigenvalues of $L_{t}$ counted with multiplicities and arranged in
increasing order. Let $\lambda $ be an eigenvalue of $L_{0}$ with
multiplicity $m$; moreover, let%
\begin{equation*}
\lambda =\lambda _{k+1}\left( 0\right) =...=\lambda _{k+m}\left( 0\right) .
\end{equation*}%
Let $U_{\lambda }$ be the eigenspace of $L_{0}$ that corresponds to the
eigenvalue $\lambda $ and $\left\{ u_{1},...,u_{m}\right\} $ be an
orthonormal basis in $U_{\lambda }$. Set for all $i,j=1,...,m$%
\begin{equation*}
Q_{ij}=\int_{M}\left. \frac{\partial V}{\partial t}\right\vert
_{t=0}u_{i}u_{j}d\mu .
\end{equation*}%
and denote by $\left\{ \alpha _{i}\right\} _{i=1}^{m}$ the sequence of the
eigenvalues of the matrix $\left\{ Q\right\} _{i,j=1}^{m}$ counted with
multiplicities and arranged in increasing order. Then we have the following
asymptotic, for any $i=1,...,m$, 
\begin{equation*}
\lambda _{k+i}(t)=\lambda _{k+i}(0)-t\alpha _{i}+o(t)\ \text{as }%
t\rightarrow 0.
\end{equation*}
\end{lemma}

The following lemma is multi-dimensional extension of \cite[Lemmas 3.4,3.6]%
{NS}. Given a connected open subset $\Omega $ of $M$ with smooth boundary,
the Dirichlet problem%
\begin{equation*}
\left\{ 
\begin{array}{l}
\Delta u=0 \,\,\,mbox{in}\,\,\Omega\\ 
u|_{\partial \Omega }=f%
\end{array}%
\right.
\end{equation*}%
has for any $f\in C\left( \partial \Omega \right) $ a unique solution that
can be represented in the form%
\begin{equation*}
u\left( y\right) =\int_{\partial \Omega }Q\left( x,y\right) f\left( x\right)
d\sigma \left( x\right)
\end{equation*}%
for any $y\in \Omega $, where $Q\left( x,y\right) $ is the Poisson kernel of
this problem and $\sigma $ is the surface measure on $\partial \Omega $. For
any $y\in \Omega $, the function $q\left( x\right) =Q\left( x,y\right) $ on $%
\partial \Omega $ will be called the Poisson kernel at the source $y$. Note
that $q\left( x\right) $ is continuous, positive and 
\begin{equation*}
\int_{\partial \Omega }qd\sigma =1.
\end{equation*}

\begin{lemma}
\label{measureMani} Let $\Omega $ be a connected open subset of $M$ with
smooth boundary and $x_{0}$ be a point in $\Omega $. Then, for any constant $%
N\geq 1$ there exists $\varepsilon =\varepsilon \left( \Omega
,N,x_{0}\right) >0$ such that for any measurable set $E\subset \Omega $ with 
\begin{equation*}
\mu \left( E\right) \leq \varepsilon
\end{equation*}%
and for any positive solution $v\in C^{2}\left( \Omega \right) $ of the
inequality 
\begin{equation}
\Delta v+Wv\geq 0\ \text{in }\Omega ,  \label{schroMani}
\end{equation}%
where 
\begin{equation}
W=\left\{ 
\begin{array}{ll}
N & \text{in \ }E, \\ 
-\frac{1}{N} & \text{in\ }\Omega \setminus E,%
\end{array}%
\right.  \label{W1}
\end{equation}%
the following inequality holds%
\begin{equation}
v(x_{0})<\int_{\partial \Omega }v\,qd\sigma ,  \label{vx0}
\end{equation}%
where $q$ is the Poisson kernel of the Laplace operator at the source $x_{0}$%
.
\end{lemma}

\begin{proof}
For any $\delta >0$ denote by $A_{\delta }$ the set of points in $\Omega $
at the distance $\leq \delta $ from $\partial \Omega $ (see Fig. \ref{pic1})
and consider the potential $V_{\delta }$ in $\Omega $ defined by%
\begin{equation}
V_{\delta }=\left\{ 
\begin{array}{l}
N\ \ \ \ \text{in }A_{\delta }, \\ 
-\frac{1}{N}\ \text{in\ }\Omega \setminus A_{\delta }.%
\end{array}%
\right.  \label{Vd}
\end{equation}%
\FRAME{ftbpF}{4.3171in}{2.2269in}{0pt}{}{\Qlb{pic1}}{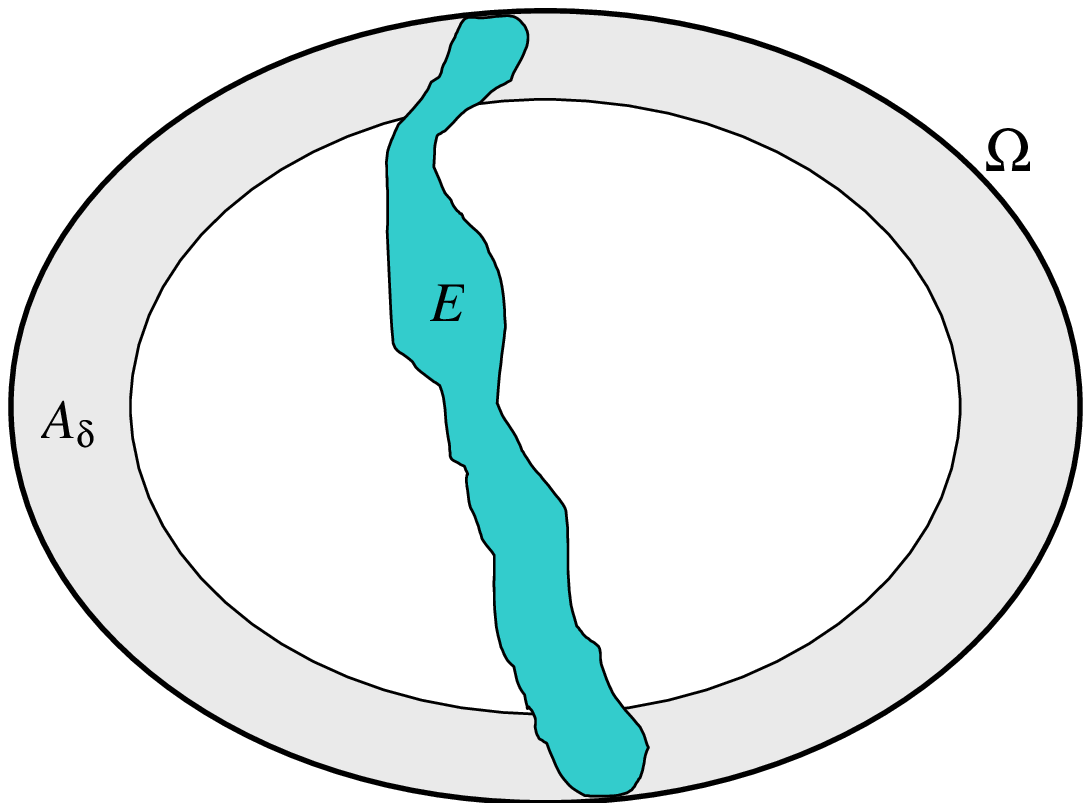}{\special%
{language "Scientific Word";type "GRAPHIC";maintain-aspect-ratio
TRUE;display "USEDEF";valid_file "F";width 4.3171in;height 2.2269in;depth
0pt;original-width 6.3027in;original-height 3.2379in;cropleft "0";croptop
"1";cropright "1";cropbottom "0";filename 'pic1.eps';file-properties
"XNPEU";}}

Since $\left\Vert V_{\delta }^{+}\right\Vert _{L^{p}\left( \Omega \right) }$
can be made sufficiently small by the choice of $\delta >0$, the following
boundary value problem has a unique positive solution:%
\begin{equation}
\left\{ 
\begin{array}{l}
\Delta w+V_{\delta }w=0\ \text{in }\Omega \\ 
w=f\ \text{on }\partial \Omega ,%
\end{array}%
\right.  \label{Wde}
\end{equation}%
for any positive continuous function $f$ on $\partial \Omega $. Denote by $%
q_{\delta }\left( x\right) $,$\ x\in \partial \Omega ,$ the Poisson kernel
of (\ref{Wde}) at the source $x_{0}$. Letting $\delta \rightarrow 0$, we
obtain that the solution of (\ref{Wde}) converges to that of%
\begin{equation}
\left\{ 
\begin{array}{l}
\Delta w-\frac{1}{N}w=0\ \text{in }\Omega \\ 
w=f\ \text{on }\partial \Omega .%
\end{array}%
\right.  \label{W0}
\end{equation}%
Denoting by $q_{0}$ the Poisson kernel of (\ref{W0}) at the source $x_{0}$,
we obtain that $q_{\delta }\searrow q_{0}$ on $\partial \Omega $ as $\delta
\searrow 0$ and, moreover, the convergence is uniform.

Let $q$ be the Poisson kernel of the Laplace operator $\Delta $ in $\Omega $%
, as in the statement of the theorem. Since any solution of (\ref{W0}) is
strictly subharmonic in $\Omega $, we obtain that $q_{0}<q$ on $\partial
\Omega $. In particular, there is a constant $\eta >0$ depending only on $%
\Omega ,N,x_{0}$ such that%
\begin{equation*}
q_{0}<\left( 1-\eta \right) q\ \text{on }\partial \Omega .
\end{equation*}%
Since the convergence $q_{\delta }\rightarrow q$ is uniform on $\partial
\Omega $, we obtain that, for small enough $\delta $ (depending on $\Omega
,N,x_{0}$),%
\begin{equation*}
q_{\delta }<\left( 1-\eta /2\right) q\ \text{on\ }\partial \Omega .
\end{equation*}%
Fix such $\delta $. Consequently, we obtain for the solution $w$ of (\ref%
{Wde}) that%
\begin{equation}
w\left( x_{0}\right) <\left( 1-\eta /2\right) \int_{\partial \Omega
}fqd\sigma .  \label{wx0}
\end{equation}

Note that the function $W$ from (\ref{W1}) can be increased without
violating  (\ref{schroMani}). Define a new potential $W_{\delta }$ by%
\begin{equation}
W_{\delta }=\left\{ 
\begin{array}{l}
N\ \ \text{in\ \ }A_{\delta }\cup E, \\ 
-\frac{1}{N}\ \ \text{in }\Omega \setminus A_{\delta }\setminus E.\ 
\end{array}%
\right.   \label{W2}
\end{equation}%
Observe that, for any $p>1$%
\begin{equation*}
\left\Vert W_{\delta }^{+}\right\Vert _{L^{p}\left( \Omega \right) }^{p}\leq
N^{p}\left( \mu \left( A_{\delta }\right) +\varepsilon \right) ,
\end{equation*}%
so that by the choice of $\varepsilon $ and further reducing $\delta $ this
norm can be made arbitrarily small. By a well-known fact (see \cite{LiebLoss}%
), if $\left\Vert \,W_{\delta }^{+}\right\Vert _{L^{p}\left( \Omega \right) }
$ is sufficiently small, then the operator $-\Delta -$\thinspace $W_{\delta }
$ in $\Omega $ with the Dirichlet boundary condition on $\partial \Omega $
is positive definite, provided $p=n/2$ for $n>2$ and $p>1$ for $n=2$.

So, we can assume that the operator $-\Delta -$\thinspace $W_{\delta }$ is
positive definite. In particular, the following boundary value problem%
\begin{equation}
\left\{ 
\begin{array}{l}
\Delta u+\,W_{\delta }u=0\ \text{in }\Omega \  \\ 
u|_{\partial \Omega }=v%
\end{array}%
\right.   \label{uv}
\end{equation}%
has a unique positive solution $u$. Comparing this with (\ref{schroMani})
and using the maximum principle for the operator $\Delta +$\thinspace $%
W_{\delta }$, we obtain $u\geq v$ in $\Omega $. Since $u=v$ on $\partial
\Omega $, the required inequality (\ref{vx0}) will follow if we prove that%
\begin{equation}
u\left( x_{0}\right) <\int_{\partial \Omega }uqd\sigma .  \label{ux0}
\end{equation}

Set $\Omega _{\delta }=\Omega \setminus A_{\delta }$ and prove that%
\begin{equation}
\sup_{\Omega _{\delta }}u\leq C\int_{\partial \Omega }ud\sigma ,
\label{supu}
\end{equation}%
for some constant $C$ that depends on $\Omega ,N,\delta ,n$. By choosing $%
\varepsilon $ and $\delta $ sufficiently small, the norm $\left\Vert
W_{\delta }\right\Vert _{L^{p}}$ can be made arbitrarily small for any $p$.
Hence, function $u$ satisfies the Harnack inequality 
\begin{equation}
\sup_{\Omega _{\delta }}u\leq C\int_{\Omega _{\delta }}ud\mu   \label{uu}
\end{equation}%
where $C$ depends on $\Omega ,N,\delta $ (see \cite{AizSim}, \cite{HanSchr}%
). Let $h$ be the solution of the following boundary value problem%
\begin{equation*}
\left\{ 
\begin{array}{l}
-\Delta h-W_{\delta }h=1_{\Omega _{\delta }}\ \text{in }\Omega  \\ 
h=0\ \text{on }\partial \Omega .%
\end{array}%
\right. 
\end{equation*}%
where $\Omega _{\delta }=\Omega \setminus A_{\delta }$. Since $\left\Vert
W_{\delta }\right\Vert _{L^{q}}$ is bounded for any $q$, we obtain by the
known a priori estimates, that%
\begin{equation*}
\left\Vert h\right\Vert _{W^{2,p}\left( \Omega \right) }\leq C\left\Vert
1_{\Omega _{\delta }}\right\Vert _{L^{p}\left( \Omega \right) },
\end{equation*}%
where $p>1$ is arbitrary and $C$ depends on $\Omega ,N,\delta ,p$ (see \cite%
{LadSolUr}). Choose $p>n$ so that by the Sobolev embedding%
\begin{equation*}
\left\Vert h\right\Vert _{C^{1}\left( \Omega \right) }\leq C\left\Vert
h\right\Vert _{W^{2,p}\left( \Omega \right) }.
\end{equation*}%
Since $\left\Vert 1_{\Omega _{\delta }}\right\Vert _{L^{p}\left( \Omega
\right) }$ is uniformly bounded, we obtain by combining the above estimates
that%
\begin{equation*}
\left\Vert h\right\Vert _{C^{1}\left( \Omega \right) }\leq C,
\end{equation*}%
with a constant $C$ depending on $\Omega ,N,\delta ,n.$ 

Multiplying the equation $-\Delta h-W_{\delta }h=1_{\Omega _{\delta }}\ $by $%
u$ and integrating over $\Omega $, we obtain%
\begin{equation*}
\int_{\Omega _{\delta }}ud\mu =\int_{\partial \Omega }\frac{\partial h}{%
\partial \nu }u~d\sigma \leq C\int_{\partial \Omega }ud\sigma 
\end{equation*}%
which together with (\ref{uu}) implies (\ref{supu}).

Let $w$ be the solution (\ref{Wde}) with the boundary condition $f=u$, that
is,%
\begin{equation*}
\left\{ 
\begin{array}{l}
\Delta w+V_{\delta }w=0\ \text{in }\Omega  \\ 
w=u\ \text{on }\partial \Omega .%
\end{array}%
\right. 
\end{equation*}%
Let us consider the difference%
\begin{equation*}
\varphi =u-w.
\end{equation*}%
Clearly, we have in $\Omega $%
\begin{equation*}
\Delta \varphi +V_{\delta }\varphi =\left( \Delta u+V_{\delta }u\right)
-\left( \Delta w+V_{\delta }w\right) =(V_{\delta }-\,W_{\delta })u
\end{equation*}%
and $\varphi =0$ on $\partial \Omega $. Denoting by $G_{V_{\delta }}$ the
Green function of the operator $-\Delta -V_{\delta }$ in $\Omega $ with the
Dirichlet boundary condition, we obtain%
\begin{equation*}
\varphi \left( x_{0}\right) =\int_{\Omega }G_{V_{\delta }}\left(
x_{0},y\right) \left( \,W_{\delta }-V_{\delta }\right) u\left( y\right) d\mu
\left( y\right) .
\end{equation*}%
Since we are looking for an upper bound for $\varphi \left( x_{0}\right) $,
we can restrict the integration to the domain $\left\{ V_{\delta }\leq
\,W_{\delta }\right\} $. By (\ref{W2}) and (\ref{Vd}) we have%
\begin{equation*}
\left\{ V_{\delta }\leq \,W_{\delta }\right\} =\left( \Omega \setminus
A_{\delta }\right) \cap \left( A_{\delta }\cup E\right) =E\setminus
A_{\delta }=:E^{\prime }
\end{equation*}%
and, moreover, on $E^{\prime }$ we have%
\begin{equation*}
\,W_{\delta }-V_{\delta }=N+\frac{1}{N}<2N,
\end{equation*}%
whence it follows that%
\begin{equation*}
\varphi \left( x_{0}\right) \leq 2N\int_{E^{\prime }}G_{V_{\delta }}\left(
x_{0},y\right) u\left( y\right) d\mu \left( y\right) .
\end{equation*}%
Using (\ref{supu}) to estimate here $u\left( y\right) $, we obtain%
\begin{equation*}
\varphi \left( x_{0}\right) \leq 2NC\left( \int_{E^{\prime }}G_{V_{\delta
}}\left( x_{0},y\right) d\mu \left( y\right) \right) \int_{\partial \Omega
}ud\sigma 
\end{equation*}%
Since $\mu \left( E^{\prime }\right) \leq \varepsilon $ and the Green
function $G_{V_{\delta }}\left( x_{0},\cdot \right) $ is integrable, we see
that $\int_{E^{\prime }}G_{V_{\delta }}\left( x_{0},\cdot \right) d\mu $ can
be made arbitrarily small by choosing $\varepsilon >0$ small enough. Choose $%
\varepsilon $ so small that%
\begin{equation*}
2NC\int_{E^{\prime }}G_{V_{\delta }}\left( x_{0},y\right) d\mu \left(
y\right) <\eta /2\inf_{\partial \Omega }q,
\end{equation*}%
which implies that%
\begin{equation*}
\varphi \left( x_{0}\right) <\eta /2\int_{\partial \Omega }uqd\sigma .
\end{equation*}%
Since by (\ref{wx0}) 
\begin{equation*}
w\left( x_{0}\right) <\left( 1-\eta /2\right) \int_{\partial \Omega
}uqd\sigma ,
\end{equation*}%
we obtain%
\begin{equation*}
u\left( x_{0}\right) =\varphi \left( x_{0}\right) +w\left( x_{0}\right)
<\int_{\partial \Omega }uqd\sigma ,
\end{equation*}%
which was to be proved.
\end{proof}

Let $V_{\max }$ be a solution of the problem (\ref{optim}). Denote by $U$
the eigenspace of $-\Delta -V_{\max }$ associated with the eigenvalue $%
\lambda _{k}\left( V_{\max }\right) =0$ assuming that $N$ is sufficiently
large.

\begin{lemma}
\label{Lemq}Fix some $c>0$ and consider the set 
\begin{equation*}
F=\left\{ V_{max}\leq -c\right\} .
\end{equation*}%
Then, for any Lebesgue point $x\in F$, then there exists a non-negative
function $q\in L^{\infty }\left( M\right) $ such that

\begin{enumerate}
\item $\int_{M}q\,d\mu =1$;

\item for any $u\in U\setminus \left\{ 0\right\} $ we have 
\begin{equation}
u^{2}(x)<\int_{M}u^{2}q\,d\mu .  \label{u<int}
\end{equation}
\end{enumerate}
\end{lemma}

\begin{proof}
Set $V=V_{\max }$. Any function $u\in U$ satisfies $\Delta u+Vu=0,$ which
implies by a simple calculation that the function $v=u^{2}$ satisfies 
\begin{equation*}
\Delta v+2Vv\geq 0.
\end{equation*}%
Next, we apply Lemma \ref{measureMani} with $J=\max (2N,\frac{1}{2c})$.
Choose $r$ so small that the density of the set $F$ in $B(x,r)$ is
sufficiently close to $1$, namely, 
\begin{equation*}
\mu \left( F\cap B\left( x,r\right) \right) >\left( 1-\varepsilon \right)
\mu \left( B\left( x,r\right) \right) ,
\end{equation*}%
where $\varepsilon =\varepsilon \left( J\right) $ is given in Lemma \ref%
{measureMani}. Since $h\leq 2N\leq J$ in $B\left( x,r\right) $ and%
\begin{eqnarray*}
\mu \left( \left\{ h>-\frac{1}{J}\right\} \cap B\left( x,r\right) \right)
&\leq &\mu \left( \left\{ h>-2c\right\} \cap B\left( x,r\right) \right) \\
&=&\mu \left( \left\{ V>-c\right\} \cap B\left( x,r\right) \right) \\
&<&\varepsilon \mu \left( B\left( x,r\right) \right) ,
\end{eqnarray*}%
all the hypotheses of Lemma \ref{measureMani} are satisfied. Let $q$ be the
function that exists by Lemma \ref{measureMani} in some small ball $B\left(
x,r\right) .$ Extending $q$ by setting $q=0$ outside $B\left( x,r\right) $
we obtain a desirable function.
\end{proof}

\subsection{Proof of main Lemma \protect\ref{crucial}}

We can now prove Lemma \ref{crucial}, that is, that $V_{\max }\geq 0$.
Consider again the set%
\begin{equation*}
F=\left\{ V_{max}\leq -c\right\} ,
\end{equation*}%
where $c>0$. We want to show that, for any $c>0$, 
\begin{equation*}
\mu (F)=0,
\end{equation*}%
which will imply the claim. Assume the contrary, that is $\mu (F)>0$ for
some $c>0$. Denote by $F_{L}$ the set of Lebesgue points of $F$. For any $%
x\in F_{L}$ denote by $q_{x}$ the function $q$ that is given by Lemma \ref%
{Lemq}. For $x\notin F_{L}$ set $q_{x}=\delta _{x}$. Then $x\mapsto q_{x}$
is a Markov kernel and, for all $x\in M$ and $u\in U$%
\begin{equation}
u^{2}\left( x\right) \leq \int_{M}u^{2}q_{x}d\mu .  \label{ule}
\end{equation}

Denote by $\mathcal{M}$ the set of all probability measures on $M$. Define
on $\mathcal{M}$ a partial order: $\nu _{1}\preceq \nu _{2}$ if and only if%
\begin{equation}
\int_{M}u^{2}d\nu _{1}\leq \int_{M}u^{2}d\nu _{2}\ \text{for all }u\in
U\setminus \left\{ 0\right\} .  \label{12}
\end{equation}%
Define $\nu _{0}\in \mathcal{M}$ by 
\begin{equation*}
d\nu _{0}=\frac{1}{\mu \left( F_{L}\right) }\mathbf{1}_{F_{L}}d\mu
\end{equation*}%
and measure $\nu _{1}\in \mathcal{M}$ by%
\begin{equation*}
\nu _{1}=\int_{M}q_{x}d\nu _{0}\left( x\right) .
\end{equation*}%
Since $\nu _{0}\left( F_{L}\right) >0$, we obtain for any $u\in U\setminus
\left\{ 0\right\} $ that%
\begin{eqnarray}
\int_{M}u^{2}d\nu _{1} &=&\int_{M}\left( \int_{M}u^{2}q_{x}d\mu \right) d\nu
_{0}\left( x\right)  \notag \\
&\geq &\int_{F_{L}}\left( \int_{M}u^{2}q_{x}d\mu \right) d\nu _{0}\left(
x\right) +\int_{M\setminus F_{L}}\left( \int_{M}u^{2}q_{x}d\mu \right) d\nu
_{0}\left( x\right)  \notag \\
&>&\int_{F_{L}}u^{2}\left( x\right) d\nu _{0}\left( x\right)
+\int_{M\setminus F_{L}}u^{2}\left( x\right) d\nu _{0}\left( x\right)  \notag
\\
&=&\int_{M}u^{2}d\nu _{0}.  \label{nu10}
\end{eqnarray}%
In particular, we have $\nu _{0}\preceq \nu _{1}$. Consider the following
subset of $\mathcal{M}$:%
\begin{equation*}
\mathcal{M}_{1}=\left\{ \nu \in \mathcal{M}:\nu \succeq \nu _{1}\right\} .
\end{equation*}

Let us prove that $\mathcal{M}_{1}$ has a maximal element. By Zorn's Lemma,
it suffices to show that any chain (=totally ordered subset) $\mathcal{C}$
of $\mathcal{M}_{1}$ has an upper bound in $\mathcal{M}_{1}$. It follows
from $\dim U<\infty $ that there exists an increasing sequence $\left\{ \nu
_{i}\right\} _{i=1}^{\infty }$ of elements of $\mathcal{C}$ such that, for
all $u\in U$,%
\begin{equation*}
\lim_{i\rightarrow \infty }\int_{M}u^{2}d\nu _{i}\rightarrow \sup_{\left\{
\nu \in \mathcal{C}\right\} }\int_{M}u^{2}d\nu .
\end{equation*}%
The sequence $\left\{ \nu _{_{i}}\right\} _{i=1}^{\infty }$ of probability
measures is $w^{\ast }$-compact. Without loss of generality we can assume
that this sequence is $w^{\ast }$-convergent. It follows that the measure%
\begin{equation*}
\nu _{\mathcal{C}}=w^{\ast }\text{-}\lim \nu _{i}\in \mathcal{M}_{1}
\end{equation*}%
is an upper bound for $\mathcal{C}$.

By Zorn's Lemma, there exists a maximal element $\nu $ in $\mathcal{M}_{1}$.
Note that the measure $\nu $ can be alternatively constructed by using a
standard balayage procedure (see e.g. \cite[Proposition 2.1, p. 250]{BlieHan}%
). Consider first the measure $\nu ^{\prime }$ defined by $\nu ^{\prime
}=\int_{M}q_{x}d\nu \left( x\right) $. It follows from (\ref{ule}) that for
any $u\in U$%
\begin{eqnarray*}
\int_{M}u^{2}d\nu ^{\prime } &=&\int_{M}\left( \int_{M}u^{2}q_{x}d\mu
\right) d\nu \\
&\geq &\int_{M}u^{2}d\nu ,
\end{eqnarray*}%
that is, $\nu ^{\prime }\succeq \nu $, in particular, $\nu ^{\prime }\in 
\mathcal{M}_{1}$. Since $\nu $ is a maximal element in $\mathcal{M}_{1}$, it
follows that $\nu ^{\prime }=\nu $, which implies the identity%
\begin{equation}
\int_{M}u^{2}d\nu =\int_{M}\left( \int_{M}u^{2}q_{x}d\mu \right) d\nu .
\label{nu=nu}
\end{equation}%
Now we can prove that $\nu \left( F_{L}\right) =0$. Assuming from the
contrary that $\nu \left( F_{L}\right) >0$, we obtain, for any $u\in
U\setminus \left\{ 0\right\} $.%
\begin{eqnarray}
\int_{M}u^{2}d\nu &=&\int_{M}\left( \int_{M}u^{2}q_{x}d\mu \right) d\nu
\left( x\right)  \notag \\
&\geq &\int_{F_{L}}\left( \int_{M}u^{2}q_{x}d\mu \right) d\nu \left(
x\right) +\int_{M\setminus F_{L}}\left( \int_{M}u^{2}q_{x}d\mu \right) d\nu
\left( x\right)  \notag \\
&>&\int_{F_{L}}u^{2}\left( x\right) d\nu \left( x\right) +\int_{M\setminus
F_{L}}u^{2}\left( x\right) d\nu \left( x\right)  \notag \\
&=&\int_{M}u^{2}d\nu ,  \label{main4}
\end{eqnarray}%
which is a contradiction. Finally, it follows from (\ref{nu10}) and $\nu \in 
\mathcal{M}_{1}$ that, for any $u\in U\setminus \left\{ 0\right\} $,%
\begin{equation*}
\int_{M}u^{2}d\nu _{0}<\int_{M}u^{2}d\nu .
\end{equation*}%
Measure $\nu $ can be approximated in $w^{\ast }$-sense by measures with
bounded densities sitting in $M\setminus F_{L}.$ Therefore, there exists a
non-negative function $\varphi \in L^{\infty }\left( M\right) $ that
vanishes on $F_{L}$ and such that%
\begin{equation*}
\int_{M}\varphi d\mu =1
\end{equation*}%
and, for any $u\in U\setminus \left\{ 0\right\} $, 
\begin{equation}
\int_{M}u^{2}\varphi _{0}d\mu <\int_{M}u^{2}\varphi d\mu  \label{u<}
\end{equation}%
where $\varphi _{0}=\frac{1}{\mu \left( F_{L}\right) }\mathbf{1}_{F_{L}}$.
Consider now the potential 
\begin{equation*}
V_{t}=V_{max}+t\varphi _{0}-t\varphi .
\end{equation*}%
We have for all $t$ 
\begin{equation*}
\int_{M}V_{t}d\mu =\int_{M}V_{\max }d\mu
\end{equation*}
and for $t\rightarrow 0$ 
\begin{equation*}
\lambda _{k}(V_{t})=\lambda _{k}(V_{\max })-t\alpha +o(t),
\end{equation*}%
where $\alpha $ is the minimal eigenvalue of the quadratic form%
\begin{equation*}
Q\left( u,u\right) =\int_{M}u^{2}\left( \varphi _{0}-\varphi \right) d\mu ,
\end{equation*}%
which by (\ref{u<}) is negative definite. Therefore, $\alpha <0$, which
together with $\lambda _{k}\left( V_{\max }\right) =0$ implies that, for all
small enough $t>0$%
\begin{equation*}
\lambda _{k}(V_{t})>0.
\end{equation*}%
Finally, let us show that $\left\vert V_{t}\right\vert \leq N$ $\mathrm{a.e.}
$ Indeed, on $F$ we have 
\begin{equation*}
V_{t}\leq -c+t\varphi _{0}<N
\end{equation*}%
for small enough $t>0$, and on $M\setminus F_{L}$ we have%
\begin{equation*}
V_{t}\leq V_{\max }-t\varphi \leq V_{\max }\leq N.
\end{equation*}%
Therefore, $V\leq N$ $\mathrm{a.e.}$ for small enough $t>0$. Similarly, we
have on $F_{L}$%
\begin{equation*}
V_{t}\geq V_{\max }+t\varphi _{0}\geq V_{\max }\geq -N
\end{equation*}%
and on $M\setminus F$%
\begin{equation*}
V_{t}\geq -c-t\varphi \geq -N
\end{equation*}%
for small enough $t>0$, which implies that $\left\vert V_{t}\right\vert \leq
N$ $\mathrm{a.e.}$ for small enough $t>0$.

Hence, we obtain that $V_{t}$ is a solution to our optimization problem (\ref%
{optim}), but it satisfies $\lambda _{k}(V_{t})>0$, which contradicts the
optimality of $V_{t}$ by Lemma \ref{Lemlak=0}.


{\footnotesize
\begin{thebibliography}{9}

\def\Revista{Revista Matem\'atica Iberoamericana}
\def\IHP{Heat kernels and analysis on manifolds, graphs, and metric spaces}
\def\JDGsurveys{Surveys in Differential Geometry}
\def\GAFA{Geom. Funct. Anal.}
\def\Russian{(in Russian)\ }
\def\appear#1{to appear in {\sl {#1}}}
\def\au#1{{\bf{#1}}, }
\def\boo#1{{\rm {``#1''}, }}
\def\inboo#1{{\sl in:} {\rm ``#1'',}\ }
\def\eng{Engl. transl.:\ }
\def\jo#1{{\sl {#1}},\ }
\def\no#1{no.{#1},\ }
\def\pa#1{\, {#1}.}
\def\pab#1{\, {#1}.}
\def\pano#1{}
\def\pbh#1{{#1},\ }
\def\preprint{pre\-print }
\def\ser#1{{#1},\ }
\def\vser#1{#1}
\def\ti#1{{#1},\ }
\def\vo#1{{\bf {#1} }\ }
\def\ya#1{(#1)\ }
\def\yab#1{#1.}
\def\yano#1{}
\def\h#1{{\accent94 #1}}
\def\edt#1{ed. {#1},\ }
\def\ISBN#1{}
\def\comment#1{}
\def\MR#1#2{}
\def\CMP#1#2{}
\def\other#1{#1}
\def\libunibie#1{}
\def\onlineres#1{}
\def\onlinemy#1{}
\def\arxiv#1{arXiv:#1,\ }

\bibitem[1]{AizSim}{\au{Aizenman M., Simon B.}
     \ti{Brownian motion and Harnack's inequality for Schr\"odinger operators}
     \jo{Comm. Pure Appl. Math.}\vo{35}\ya{1982}\pa{203-271}
  }
\bibitem[2]{BirSol}{\au{Birman M.Sh., Solomyak M.Z.}
     \ti{Estimates for the number of negative eigenvalues of the Schr\"odinger operator
         and its generalizations}
     \jo{Advances in Soviet Math.}\vo{7}\ya{1991}\pa{1-55}
  }

\bibitem[3]{BlieHan}{\au{Bliedtner J., Hansen W.}
    \boo{Potential theory -- an analytic and probabilistic approach to balayage}
    \ser{Universitext} \pbh{Springer, Berlin-Heidelberg-New York-Tokyo}
    \yab{1986}\MR{88b}{31002}
  }
\bibitem[4]{Soufi}{\au{El Soufi, Ahmad}
    \ti{Isoperimetric inequalities for the eigenvalues of natural Schr{\"o}dinger operators on surfaces}
    \jo{Indiana Univ. Math. J.}\vo{58}\ya{2009}\no{1}\pa{335-349}
  }
\bibitem[5]{GrigNadirNeg}{\au{Grigor'yan A., Nadirashvili N.}
     \ti{Negative eigenvalues of two-dimensional Schrödinger equations}
     \other{arXiv:1112.4986}\yano{2012}\comment{neg}
  }
\bibitem[6]{GNY}{\au{Grigor'yan A., Netrusov Yu., Yau S.-T.}
    \ti{Eigenvalues of  elliptic operators and geometric applications}
    \inboo{Eigenvalues of Laplacians and other geometric operators}
    \ser{Surveys in Differential Geometry \vser{IX}}\ya{2004}\pa{147-218}
    \comment{eigsch}\comment{esch}\comment{GrigYauSch}
  }
\bibitem[7]{HanSchr}{\au{Hansen W.}
    \ti{Harnack inequalities for Schr\"odinger operators}
    \jo{Ann. Scuola Norm. Sup. Pisa}\vo{28}\ya{1999}\pa{413-470}
   }
 \bibitem[8]{Harrell}{\au{Harrell II, E.M} 
     \ti{On the second eigenvalue of the Laplace operator penalized by curvature} 
     \jo{Diff. Geom. Appl.} \vo{6}\ya{1996} \pa{397-400} 
     \comment{Potential V=ak1k2+b(k1²+k2²)}
  }
\bibitem[9]{Kato}{\au{Kato T.}
        \boo{Perturbation theory for linear operators}
        \pbh{Springer}
        \yab{1995}   
  }
\bibitem[10]{LadSolUr}{\au{Ladyzenskaja O.A., V.A. Solonnikov, Ural'ceva N.N.}
    \boo{Linear and quasilinear equations of parabolic type}
    \pbh{Providence, Rhode Island}\yab{1968}
  }
\bibitem[11]{LiYauE}{\au{Li P., Yau S.-T.}
    \ti{On the Schr\"odinger equation and the eigenvalue problem}
    \jo{Comm. Math. Phys.}\vo{88}\ya{1983}\pa{309--318}
  }  
\bibitem[12]{Lieb}{\au{Lieb E.H.}
     \ti{The number of bound states of one-body Schr\"odinger operators and the Weyl problem}
     \jo{Proc. Sym. Pure Math.}\vo{36}\ya{1980}\pa{241-252}
  }
\bibitem[13]{LiebLoss}{\au{Lieb E.H., Loss M.}
     \boo{Analysis}
     \pbh{AMS}
     \yab{2001}
  }  
\bibitem[14]{NS}{\au{Nadirashvili N., Sire Y.} 
    \ti{Conformal spectrum and harmonic maps}  
    \other{arXiv:1007.3104}\yano{2013} 
  }
\bibitem[15]{YangYau}{\au{Yang P., Yau S.-T.}
     \ti{Eigenvalues of the Laplacian of compact Riemann surfaces and minimal submanifolds}
     \jo{Ann. Scuola Norm. Sup. Pisa Cl. Sci. (4)}\vo{7}\ya{1980}\pa{55-63}
  }

\end{thebibliography}
}
\end{document}